\documentclass{amsart}
\usepackage{amsmath,amssymb,amsthm,amscd}
\usepackage{enumerate}
\usepackage{enumerate}
\newtheorem{thm}{Theorem}[section]
\newtheorem*{main}{Main Theorem}
\newtheorem{prop}[thm]{Proposition}
\newtheorem{lem}[thm]{Lemma}
\newtheorem{cor}[thm]{Corollary}
\theoremstyle{definition}
\newtheorem{df}[thm]{Definition}

\theoremstyle{remark}

\newcommand{\cO}{\mathcal O}

\newcommand{\Hom}{\operatorname{\rm Hom}}

\newcommand{\fg}{\mathfrak g}
\newcommand{\fh}{\mathfrak h}

\newcommand{\fb}{\mathfrak b}

\newcommand{\dD}{\widehat \Delta}
\newcommand{\de}{\widehat \varepsilon}
\newcommand{\dS}{\widehat S}

\newcommand{\frsl}{\mathfrak{sl}}

\newcommand{\ch}{\operatorname{\mathrm ch}}

\begin{document}
	\title[Comparison of unitary duals]{Comparison of unitary duals of Drinfeld doubles and complex semisimple Lie groups}
	\author{Yuki Arano}
	\address{Department of Mathematical Sciences, the University of Tokyo, Komaba, Tokyo, 153-8914, Japan}
	\email{arano@ms.u-tokyo.ac.jp}
	\date{}
	\begin{abstract}
	We determine a substantial part of the unitary representation theory of the Drinfeld double of a $q$-deformation of a compact Lie group in terms of the complexification of the compact Lie group. Using this, we show that the dual of every $q$-deformation of a higher rank compact Lie group has central property ${\rm (T)}$. We also determine the unitary dual of $SL_q(n,\mathbb{C})$.
	\end{abstract}
	\maketitle
	\section{Introduction}
	In our previous paper \cite{Arano}, we have studied the unitary representation theory of quantum doubles of $q$-deformations of compact Lie groups and proved property ${\rm (T)}$ for the double of $SU_q(2n+1)$ (which is equivalent to central property ${\rm (T)}$ for the dual of $SU_q(2n+1)$).
	Regarding such doubles as such doubles a $q$-deformation of the complexification of the Lie group served as the main idea of our previous work.

	In \cite{JL}, Joseph and Letzter defined a notion of quantum Harish-Chandra module, which also can be seen as a certain representation of a $q$-deformation of a complex semisimple Lie group.
	In this paper, we compare these two notions and show that the quantum Harish-Chandra modules are nothing but the admissible representations of quantum doubles, which already has been implicitly prospected in \cite{VY}. Then we use deep analysis on quantum Harish-Chandra modules to compare the representation theory of the quantum doubles and that of the classical case.
	Our main theorem is as follows. Recall that in \cite{Arano}, we have shown that every irreducible unitary representation is admissible.
	\begin{main}
	Let $K$ be a connected simply connected compact Lie group and fix $0<q<1$.
	Consider the $q$-deformation $K_q$. Let $Q$ (resp.\ $P$) be the root lattice (resp.\ the weight lattice) and $W$ the Weyl group.
	\begin{enumerate}
	\item The irreducible admissible representations of the quantum double $G_q$ of $K_q$ are parametrized by $(P \times X)/W$, $X = \fh^*/ 2 \pi i \log(q)^{-1} Q^\vee$
 and $W$ acts on $P \times X$ by the diagonal action.
	\item For $(\lambda,\nu) \in P \times \fh^*$ such that ${\rm Im}(\nu)$ is small enough, the corresponding irreducible admissible representation of $G_q$ is unitary if and only if the corresponding irreducible representation of the complexification $G$ of $K$ is unitary.
	\end{enumerate}
	\end{main}
	This result allows us to: 
	\begin{itemize}
	\item classify a substantial amount of unitary representations of such doubles in terms of those of complex semisimple Lie groups,
	\item prove central property ${\rm (T)}$ for the duals of general $q$-deformations with rank equal or larger than $2$ and
	\item classify all unitary representations of the quantum double of $SU_q(n)$.
	\end{itemize}
	In particular, this approach generalizes our previous result \cite{Arano} and also the one by Jones \cite{CJones} in the quantum $G_2$-case.

	This work is motivated by the theory of operator algebras as follows. The study of central multipliers on compact quantum groups has been started by De Commer, Freslon and Yamashita \cite{DCFY}. It was already implicitly present in the study of approximation properties by Brannan \cite{Brannan} and Freslon \cite{Freslon}. In \cite{DCFY}, it is also shown that the central multipliers are the same as the multipliers of quantum doubles and therefore have a strong relationship with the unitary representation theory of quantum doubles, which was already studied by Pusz \cite{Pusz} and Voigt \cite{Voigt} in the case of $SU_q(2)$.

	There is also a strong relationship with the theory of subfactors. Popa and Vaes \cite{PV}  introduced a notion of multipliers for tensor categories, which appeared to be the same as the corresponding central multipliers in the quantum group case and also the multipliers for standard invariants in the case of subfactors in \cite{Popa}. Neshveyev-Yamashita \cite{NY} and Ghosh-Jones \cite{GJ} also introduced equivalent notions in different approaches. Together with central property ${\rm (T)}$ of quantum groups, this notion eventually lead us to the a first example of non group-like subfactors with property ${\rm (T)}$ standard invariant.
	
	{\bf Acknowledgment} The author wishes to express his gratitude to Hironori Oya and Kohei Yahiro for many fruitful discussions.
	He is grateful to Tim de Laat and Jonas Wahl for many valuable comments on the draft version of this manuscript.
	He wishes to thank KU Leuven, where the paper was written, for the invitation and hospitality.
	He also appreciates the supervision of Yasuyuki Kawahigashi.
	This work was supported by the Research Fellowship of the Japan Society for the Promotion of Science and the Program for Leading Graduate Schools, MEXT, Japan.
\section{Preliminaries}
	We repeat the definitions in \cite{Arano}, except for the conventions that we write $G$ for the complexification and $K$ for the maximal compact subgroup.

	Let $K$ be a connected simply connected compact Lie group.
	Take its complexification $G = K_\mathbb{C}$ with its Iwasawa decomposition $G=KAN$. Take the Lie algebra $\fg$ of $G$ and a Cartan algebra $\fh$. Let $(\cdot,\cdot)$ be the natural bilinear form on $\fh$, which is normalized as $(\alpha,\alpha)=2$ for a short root $\alpha$. For each $\alpha \in \Delta$, let $\alpha^\vee := 2\alpha/(\alpha,\alpha)$ be the coroot.
	Take the set of roots $\Delta \subset \fh^*$, the (co)root lattice $Q$ ($Q^\vee)$ and the (co)weight lattice $P$ ($P^\vee$). Fix a set of simple roots $\Pi \subset \Delta$ and let $Q_+$, $Q^\vee_+$, $P_+$ and $P^\vee_+$ be the positive part of the corresponding lattices. 
	Put $q_\alpha := q^{(\alpha,\alpha)/2}$,
	\[n_q := \frac{q^n - q^{-n}}{q - q^{-1}},\]
	\[n_q ! := n_q (n-1)_q \dots 1_q,\]
	\[\left( \begin{matrix} n \\ m \end{matrix} \right)_q := \frac{n_q!}{m_q! (n-m)_q!}.\]
	\begin{df}
	The \emph{quantized enveloping algebra} $U_q(\fg)$ is the Hopf $*$-algebra generated by $\{K_\lambda, E_\alpha, F_\alpha \mid \lambda \in P, \alpha \in \Pi\}$ with the relations
	\[K_0 = 1, \quad K_\lambda K_\mu = K_{\lambda + \mu},\]
	\[K_\lambda E_\alpha K_{-\lambda} = q^{(\alpha,\lambda)} E_\alpha,\quad K_\lambda F_\alpha K_{-\lambda} = q^{-(\alpha,\lambda)} F_\alpha,\]
	\[[E_\alpha,F_\beta] = \delta_{\alpha,\beta} \frac{K_\alpha - K_{-\alpha}}{q_\alpha - q_\alpha^{-1}},\]
	\[\sum_{r=0}^{1-(\alpha,\beta^\vee)} (-1)^r \left( \begin{matrix} 1-(\alpha,\beta^\vee) \\ r \end{matrix} \right)_{q_\beta} E_\beta^r E_\alpha E_\beta^{1-(\alpha,\beta^\vee)-r} = 0,\]
	\[\sum_{r=0}^{1-(\alpha,\beta^\vee)} (-1)^r \left( \begin{matrix} 1-(\alpha,\beta^\vee) \\ r \end{matrix} \right)_{q_\beta} F_\beta^r F_\alpha F_\beta^{1-(\alpha,\beta^\vee)-r} = 0,\]
		\[K_\lambda^* = K_\lambda,\quad E_\alpha^* = F_\alpha K_\alpha,\quad F_\alpha^* = K_{-\alpha}E_\alpha,\]
	\[\dD(K_\lambda) = K_\lambda \otimes K_\lambda, \quad \de(K_\lambda) = 1, \quad \dS(K_\lambda) = K_{-\lambda},\]
	\[\dD(E_\alpha) = E_\alpha \otimes 1 + K_\alpha \otimes E_\alpha,\quad \de(E_\alpha) = 0,\quad \dS(E_\alpha) = -K_{-\alpha} E_\alpha,\]
	\[\dD(F_\alpha) = F_\alpha \otimes K_{-\alpha} + 1 \otimes F_\alpha,\quad \de(F_\alpha)=0,\quad \dS(F_\alpha) = -F_\alpha K_\alpha.\]
	\end{df}
	Let $U_q(\fh)$ (resp.\ $U_q(\fb^+)$, $U_q(\fb^-)$) be the subalgebra generated by $K_\lambda$ (resp.\ $K_\lambda$ and $E_\alpha$, $K_\lambda$ and $F_\alpha$).
	
	For each $\lambda \in \fh^*$, let $M(\lambda)$ be the Verma module of highest weight $\lambda$ and $V(\lambda)$ its unique irreducible quotient.
	If $\lambda \in P_+$, then $V(\lambda)$ is finite dimensional. We say a $U_q(\fg)$-module is of {\it type 1} if it decomposes into a direct sum of $V(\lambda)$'s for $\lambda \in P_+$. Notice that any subquotient of type 1 module is also of type 1.
	\begin{df}
	The \emph{quantum coordinate algebra} $\cO(K_q) \subset U_q(\fg)^*$ is the subspace of matrix coefficients of type 1 representations. Then $\cO(K_q)$ carries a unique Hopf $*$-algebra structure which makes the pairing $\cO(K_q) \times U_q(\fg) \to {\mathbb C}$ skew, that is,
	\[(ab,c) = (a \otimes b, \dD(c)),\]
	\[(a,cd) = (\Delta(a), d \otimes c),\]
	\[(1,c) = \varepsilon(c),\]
	\[(a,1) = \de(a),\]
	\[(a^*,c) = \overline{(a,\dS(b)^*)}.\]
	for $a,b \in \cO(K_q)$, $c,d \in U_q(\fg)$.
	\end{df}
	We also use the following notation:
	for $a \in \cO(K_q)$ and $b\in U_q(\fg)$
	\begin{align*}
	a \triangleright b & := (a,b_{(2)}) b_{(1)}, & b \triangleleft a & := (a,b_{(1)}) b_{(2)}, \\
	b \triangleright a & := (a_{(1)},b) a_{(2)}, & a \triangleleft b & := (a_{(2)},b) a_{(1)},
	\end{align*}
	where we used the sumless Sweedler notation:
	\[\Delta(x) = x_{(1)} \otimes x_{(2)}.\]

	Since we need to deal with all (possibly) non-real weights in $\fh^*$, we use some terminologies, which are used only in this article.
	\begin{df}
	We say that $\nu \in \fh^*$ is \emph{dominant} (with respect to $q$) if $(\nu,\alpha^\vee) \not \in \mathbb{Z}_{<0} + 2 \pi i \log(q_\alpha)^{-1} \mathbb{Z}$.
	
	We say that $\nu \in \fh^*_\mathbb{R}$ is \emph{small} if $(\nu,\alpha)<1$ for any $\alpha \in \Delta$.

	We say that $\nu \in \fh^*$ is \emph{almost real} (with respect to $q$) if $\frac{\log(q)}{2\pi} {\rm Im}(\nu)$ is small.
	\end{df}
	Remark that the set of small (resp.\ almost real) weights is open and invariant under the Weyl group action.
	For the later use, we state several lemmas.
	The following argument was suggested by Hironori Oya.
	\begin{lem}\label{lem case-by-case}
	Let $\mu,\nu \in \fh^*_{\mathbb R}$ be small. Then $\mu - \nu \in Q^\vee$ implies $\mu = \nu$.
	\end{lem}
	\begin{proof}
	Since $(\mu-\nu,\alpha) < 2$ for any $\alpha \in \Delta$, it suffices to show the following:

	For any $x \in Q^\vee$, $(x,\alpha) < 2$ for all $\alpha \in \Delta$ implies $x = 0$.

	To show this, after conjugating by the Weyl group action if necessary, we may assume $x \in Q^\vee_+$ and has a minimal height among $Wx \cap Q^\vee_+$. Then, since $x \in Q^\vee_+$, there exists $\alpha \in \Delta_+$ such that $(x,\alpha)>0$. Since $(x,\alpha)<2$, we get $(x,\alpha)=1$. This asserts $s_\alpha(x) = x - \alpha^\vee$.

	Now, since we have assumed that $x$ has a minimal height, we get $x - \alpha^\vee \not \in Q^\vee_+$, which means
	\[x = \sum_{\beta \in \Pi, \beta \neq \alpha} n_\beta \beta^\vee.\]
	In particular, $(x,\alpha) < 0$, which is a contraction.
	\end{proof}
	All the extraordinariness of the type $A$ case in this paper comes from the following easy lemma.
	\begin{lem}\label{lem type A}
	Let $K=SU(n)$.
	Then for any $\nu \in \fh^*_\mathbb{R}$, there exists $\lambda \in P^\vee$ such that $\nu-\lambda$ is small.
	\end{lem}
	\begin{proof}
	We identify $\fh^*_\mathbb{R} \simeq \mathbb{R}^n/\mathbb{R}(1,1,\dots,1)$ with the weight lattice $\mathbb{Z}^n/\mathbb{Z}(1,1,\dots,1)$. Write $\nu = (\nu_1,\nu_2,\dots,\nu_n) \in \fh^*_\mathbb{R}$. Let $\lambda_i$ be the integer such that $0 \leq \nu_i - \lambda_i < 1$. Then $\lambda = (\lambda_i)$ is the desired element in $P = P^\vee$.
	\end{proof}
	For the later use, we consider the subalgebra $J$ of $U_q(\fg)$ generated by $K_{2 \lambda}, E_\alpha, F_\alpha K_\alpha$. Remark that $J$ is the localization of the adjoint finite part $F(U_q(\fg))$ with respect to the Ore set $\{K_{-2 \lambda} \mid \lambda \in P_+\}$.
	Consider the category $\cO$ over $J$. Then the weight makes sense as a element in $\fh^*/\pi i \log(q)^{-1} Q^\vee = \frac{1}{2} X$. For each $\Lambda \in \fh^*$, we put
	\[\cO^\Lambda := \{M \in \cO \mid {\rm wt}(M) \subset \Lambda + P\}.\]

	The following lemma may be well-known to experts, but we could not find any references.
	\begin{lem}\label{lem char}
	When $q$ is not a root of unity and any $\lambda \in \fh^*$ such that $2\lambda$ is almost real, we have $\ch V_q(\lambda) = \ch V_1(\lambda)$.
	\end{lem}
	\begin{proof}
	Thanks to \cite[Corollary 4.8]{EK}, we have this equality for generic $q$.

	In general, let us recall that we have an invariant form $S_q$ called the Shapovalev form on $M_q(\lambda)$, which depends algebraically on $q$ and
	\[V_q(\lambda) = M_q(\lambda)/{\rm Ann}(S_q).\]
	Then, \cite[Theorem 4.1.16]{Joseph} shows the order of zeros of the determinant of $S_q^\mu = S_q|_{M_q(\lambda)_\mu}$ is constant along $[q,1]$, hence we get
	\[\dim V_q(\lambda)_\mu \leq \dim V_1(\lambda)_\mu.\]

	The converse inequality is shown in \cite[Proposition 6.1]{AM}.
	\end{proof}
	For $\lambda \in P_+$, we regard $V_q(\lambda)$ as a family of representation of $U_q(\fg)$ on a single vector space $V(\lambda)$ such that weight spaces are the same and each $U_{q^\alpha}(\frsl_{2,\alpha})$-isotypical components varies continuously with respect to $q$. (This is possible, for example, via the global base.)
	\section{Admissible representations}\label{sec para}
	Let us recall the following definition from \cite{Arano}.
	\begin{df}
	The \emph{Drinfeld double} $D$ of $K_q$ is a $*$-algebra generated by $\cO(K_q)$ and $U_q(\fg)$ with the commutation relation
	\[ab = (a_{(1)} \triangleright b \triangleleft S(a_{(3)})) a_{(2)}\]
	for $a \in \cO(K_q)$, $b \in U_q(\fg)$.
	
	A representation of $D$ is said to be \emph{admissible} if it is of type 1 as a representation of $U_q(\fg)$ and the multiplicity of each irreducible representation of $U_q(\fg)$ is finite.  The multiplicity as a $U_q(\fg)$-representation is called \emph{$K$-type multiplicity}.
	\end{df}
	Let ${\rm Adm}(G_q)$ be the category of admissible representations of $D$.
	
	For $0<q<1$, let us recall the construction of the (nonunitary) principal series $L_q(\lambda,\nu)$ of $G_q$ which essentially appeared in \cite[Section 4]{Arano}. 
	For each $\lambda \in P$ and $\nu \in \fh^*$, we take the $1$-dimensional representation $\mathbb{C}_{(\lambda,\nu)}$ of $B = U_q(\fh) \cO(K_q) \subset D$ given by
	\[K_\mu \mapsto q^{(\lambda,\mu)}, x \mapsto K_\nu(x)\]
	for $\mu \in P$, $x \in \cO(K_q)$. Then
	\[L_q(\lambda,\nu) := D_c \otimes_B \mathbb{C}_{(\lambda,\nu-2\rho)}.\]
	Hence, as $U_q(\fg)$-module, $L_q(\lambda,\nu) \simeq L_\lambda = \bigoplus_{\mu \in P_+} V(\mu) \otimes V(\mu)^*_\lambda$.

	For $q=1$, let $L_1(\lambda,\nu)$ be the Harish-Chandra module of the (nonunitary) principal series of $G$, that is, the $K$-finite part of
	\[\{f \in C^\infty(G) \mid f(gtan) = a^{\frac{1}{2} \nu - \rho} f(g)\}\]
	for $t \in T$, $a \in A$ and $n \in N$.
	Again $L_1(\lambda,\nu) \simeq L_\lambda$ as $K$-module.
	
	For each $0< q \leq 1$, from the exact same argument as in \cite[Proposition 5.3]{Arano}, we have a nondegenerate invariant sesquilinear form
	\[(\cdot,\cdot)_q^0: L_q(\lambda,\nu) \times L_q(\lambda,-\overline{\nu}) \to {\mathbb C}.\]
	From the construction, this sesquilinear form varies continuously with respect to $q$ (identifying $L_q(\lambda,\nu) = L_\lambda$).
	
	In \cite{VY}, it is shown that for each $w \in W$, there exists a meromorphic family of intertwining operators
	\[T_q^w: L_q(\lambda,\nu) \to L_q(w\lambda,w\nu).\]
	For a simple reflection $w=s_\alpha$, $T_q^\alpha := T_q^{s_\alpha}$ is given by the following.
	
	For each $v \in L_q(\lambda,\nu)$ such that $v \in V(\lambda) \otimes V(s)^*_{m}$ as an $\frsl_{2,\alpha}$-module,
	\[T_q^w v = \prod^s_{k=|m|+1} \frac{(k-z)_{q_\alpha}}{(k+z)_{q_\alpha}} v,\]
	where $z = \frac{1}{2}(\nu,\alpha^\vee)$, $m = \frac{1}{2}(\lambda,\alpha^\vee)$.

	This formula tends to the one in the classical case (\cite[Proposition 3.7]{Duflo})
	\[T_1^w v = \prod^s_{k=|m|+1} \frac{k-z}{k+z} v,\]
	so that we may form $T_q^w$ as a continuous family with respect to $q$, as long as the denominator is nonzero (in particular, if $\nu$ is almost real and $-\frac{1}{2}(\lambda - \nu) - \rho$ is dominant).
	\section{Equivalence of categories}
	Recall that $U_q(\fb^+)$ (resp.\ $U_q(\fb^-)$) is the subalgebra of $U_q(\fg)$ generated by $K_\lambda$ and $E_\alpha$ (resp.\ $K_\lambda$ and $F_\alpha$).
	Define Hopf algebras
	\[\cO(B_q^\pm) := \cO(K_q)/{\rm Ann}(U_q(\fb^{\mp})).\]
	Then the maps
	\begin{align*}
	l_+:\cO(B_q^+) \to U_q(\fb^+): x \mapsto (x \otimes {\rm id})(R), \\
	l_- :\cO(B_q^-) \to U_q(\fb^-): x \mapsto ({\rm id} \otimes x)(R^{-1})
	\end{align*}
	are isomorphisms of Hopf algebras.
	We put
	\[\Psi(x)= (l_- \otimes l_+)\Delta(x) = ({\rm id} \otimes x \otimes {\rm id})(R_{23} R_{12}^{-1}).\]
	Then
	\[I(x) := l_-(x_{(1)})\dS^{-1}(l_+(x_{(2)})) =({\rm id} \otimes x)(R_{12}^{-1} R_{21}^{-1})\]
	is a $U_q(\fg)$-module isomorphism from $\cO(K_q)$ onto $F(U_q(\fg))$.

	We have an algebra embedding
	\[\dD \times \Psi: D \to U_q(\fg) \otimes U_q(\fg).\]
	We describe the image of this map.
	Put
	\[D' := (F(U_q(\fg)) \otimes \dS(F(U_q(\fg))))\dD(U_q(\fg)).\]
	\begin{prop} We have
	\[(\dD \times \Psi)(D) = D'.\]
	\end{prop}
	\begin{proof}
	First, notice that
	\[D' = (F(U_q(\fg)) \otimes 1) \dD(U_q(\fg)).\]
	We also remark that
	\[\beta:U_q(\fg) \otimes U_q(\fg) \to U_q(\fg) \otimes U_q(\fg) : x \otimes y \mapsto (x \otimes 1)\dD(y)\]
	is an isomorphism of vector spaces with inverse map
	\[\beta^{-1}(x \otimes y) = x \dS(y_{(1)}) \otimes y_{(2)}.\]
	Combining these, it suffices to show that
	\[\beta^{-1}(\dD \times \Psi)(D) = F(U_q(\fg)) \otimes U_q(\fg).\]
	
	Now using
	\[m({\rm id} \otimes \dS)\Psi(x) = I(x),\]
	we get
	\begin{align*}
	\beta^{-1}(\dD \times \Psi)(a x) & = \beta^{-1}(l_+(a_{(1)}) \otimes l_-(a_{(2)}))(1 \otimes x) \\
	& = l_+(a_{(1)}) \dS(l_-(a_{(2)})) \otimes l_-(a_{(3)}) x \\
	& = I(a_{(1)}) \otimes l_-(a_{(2)}) x. 
	\end{align*}
	for $a \in \cO(K_q)$ and $x \in U_q(\fg)$. Hence we get the conclusion.
	\end{proof}
	Now we see that the quantum Harish-Chandra module by Joseph and Letzter is nothing but the admissible representation of $D$, and hence we can apply the categorical equivalence between the quantum Harish-Chandra modules and the category ${\mathcal O}$.

	Let $\kappa$ be the involutive antiautomorphism on $U_q(\fg)$ defined by
	\[\kappa(E_\alpha) = K_{-\alpha} F_\alpha, \kappa(F_\alpha) = E_\alpha K_\alpha, \kappa(K_\lambda) = K_\lambda.\]	
	For $\Lambda \in \fh^*$ and $V \in \cO^\Lambda_f$, we define a $U_q(\fg) \otimes U_q(\fg)$-module structure on $(M(\Lambda) \otimes V)^*$ by
	\[(v, xf) := ((\kappa \otimes \dS^{-1})(x) v, f)\]
	for $x \in U_q(\fg) \otimes U_q(\fg)$, $v \in M(\Lambda) \otimes V$ and $f \in (M(\Lambda) \otimes V)^*$.
	Let $\Psi_\Lambda(V)$ be the finite part with respect to the action of $\Delta(U_q(\fg))$. Then $\Psi_\Lambda(V)$ is a $D'$-module in a natural fashion. Via the isomorphism $(\dD \times \Psi)^{-1}$, we regard $\Psi_\Lambda(V)$ as a $D$-module.

	This identification gives a decomposition ${\rm Adm}(G_q) = \bigoplus_{\Lambda \in \frac{1}{2}X/W} {\rm Adm}(G_q)_\Lambda$ with respect to the central character of the right $F(U_q(\fg))$-action.
	The following theorem is a direct translation of \cite[Section 8.4]{Joseph} in our setting.
	\begin{thm}\label{thm JL}
	For every dominant weight $\Lambda \in \fh^*$, we have the following.
	\begin{enumerate}[(i)]
	\item We have a contravariant exact functor
	\[\Psi_\Lambda: \cO^\Lambda_f \to {\rm Adm}(G_q)_\Lambda: V \mapsto F(M(\Lambda) \otimes V)^*.\]
	\item There exists a contravariant functor $T_\Lambda: {\rm Adm}(G_q)_\Lambda \to \cO^\Lambda_f$ such that $\Psi_\Lambda \circ T_\Lambda = {\rm id}$ and
	\[\Hom(V,T_\Lambda(X)) \simeq \Hom(X,\Psi_\Lambda(V)).\]
	\item If $\Lambda$ is regular, the functor $\Psi_\Lambda$ is a categorical equivalence.
	\item For every irreducible $V \in \cO^\Lambda_f$, the admissible module $\Psi_\Lambda(V)$ is either $0$ or irreducible. Moreover this exhausts all irreducibles.
	\end{enumerate}
	\end{thm}
	We use the following lemma, which is proven in \cite{VY}.
	\begin{lem}
	We have an isomorphism
	\[\Psi_\Lambda(M_q(\Lambda')) \simeq L_q(\lambda,\nu),\]
	where $(\lambda,\nu) = (\Lambda - \Lambda', - \Lambda - \Lambda' - 2 \rho)$.
	\end{lem}
	We use the dot-action of the Weyl group on $\fh^*$:
	\[w.\lambda := w(\lambda+\rho)-\rho.\]
	\begin{lem}\label{lem Weyl group}
	For a dominant weight $\Lambda$, $\alpha \in \Delta_+$ and $\Lambda'$ such that $(\Lambda' + \rho,\alpha^\vee) \geq 0$,
	\[\| \Lambda - \Lambda' \| \leq \|\Lambda - s_\alpha. \Lambda' \|,\]
	where $\| \lambda \|$ is the square root of $(\lambda,\lambda)$.
	Moreover the equality holds if and only if $s_\alpha$ stabilizes either $\Lambda$ or $\Lambda'$.
	\end{lem}
	\begin{proof}
	Immediate from  $\| \Lambda - s_\alpha. \Lambda'\|^2 - \|\Lambda-\Lambda'\|^2 = 4 (\Lambda'+\rho,\alpha^\vee)(\Lambda+\rho,\alpha)$.
	\end{proof}
	\begin{cor}
	We have the following.
	\begin{enumerate}
	\item The module $\Psi_\Lambda(V_q(\Lambda'))$ is isomorphic to $V_q(\lambda,\nu)$ if and only if there is no $\alpha \in \Delta_+$ such that
	\[(\Lambda+\rho,\alpha^\vee) = 0, (\Lambda'+\rho,\alpha^\vee) \in \mathbb{Z}_{\geq 0}.\]
	Otherwise $\Psi_\Lambda(V_q(\Lambda')) = 0$.
	\item The set of modules $\{V_q(\lambda,\nu) \mid \lambda \in P, \nu \in X\}$ exhausts all irreducibles.
	 \item We have $V_q(\lambda,\nu) \simeq V_q(\lambda',\nu')$ if and only if there exists $w \in W$ such that $(\lambda,\nu) = (w \lambda', w \nu')$.
	\item If $\frac{1}{2}(\lambda-\nu) - \rho$ is dominant, then $L_q(\lambda,\nu)$ contains $V_q(\lambda,\nu)$ as a submodule.
	\item If $-\frac{1}{2}(\lambda-\nu) - \rho$ is dominant, then $V_q(\lambda,\nu)$ is a quotient of $L_q(\lambda,\nu)$.
	\end{enumerate}
	\end{cor}
	\begin{proof}
	For (1), assume there is no $\alpha \in \Delta_+$ as above. Take a resolution of $V_q(\Lambda')$ by Verma modules.
	Then the Verma modules appearing in the resolution are of the form $M_q(w. \Lambda')$ such that
	\[w = s_{\alpha_1} s_{\alpha_2} \ldots s_{\alpha_k},\]
	\[\Lambda' \geq s_{\alpha_k} .\Lambda' \geq s_{\alpha_{k-1}} s_{\alpha_k}.\Lambda' \geq \ldots \geq s_{\alpha_1} s_{\alpha_2} \ldots s_{\alpha_k}. \Lambda' = w. \Lambda'\]
	for $\alpha_i \in \Delta_+$.
	
	Now by assumption, iterated use of Lemma \ref{lem Weyl group} shows that
	\[\|\Lambda- w.\Lambda'\| > \|\Lambda - \Lambda'\|.\]
	We apply $\Psi_\Lambda$ to this resolution to get a resolution of $V_q(\lambda,\nu)$ by principal series representations. Then, the above estimate shows that the minimal $K$-type of principal series representations appearing in the resolution is always strictly larger than $|\lambda|$, where $|\lambda|$ is the unique dominant element in the Weyl group orbit of $\lambda$. Hence the minimal $K$-type of $\Psi_\Lambda(V(\Lambda'))$ is $|\lambda|$. In particular, it is nonzero and hence irreducible, by (iv) of Theorem \ref{thm JL}. 

	Conversely, assume there is such $\alpha$. Then $V_q(\Lambda')$ is a quotient of the cokernel of $M_q(s_\alpha.\Lambda') \to M_q(\Lambda')$, which is injective. Applying $\Psi_\Lambda$, we get that $\Psi_\Lambda(V_q(\Lambda'))$ is a submodule of the kernel of the surjective map $L_q(\lambda,\nu) \to L_q(s_\alpha \lambda, s_\alpha \nu) \simeq L_q(\lambda, s_\alpha \nu)$. Now the comparison of the $K$-type multiplicity gives that this map has to be injective.

	(2) follows from (iv) of Theorem \ref{thm JL} and (1).
	
	For (3), we use the functor $T_\Lambda$ in (ii) of Theorem \ref{thm JL}.
	We only need to show that $H := \Psi_\Lambda(V(\Lambda')) \simeq \Psi_\Lambda(V(\Lambda''))$ implies $\Lambda = \Lambda''$. For this, assume $\Lambda \neq \Lambda'$ and put $M:= T_\Lambda(H)$. Then the adjoint property gives injections $V(\Lambda) \to M$ and $V(\Lambda') \to M$. This gives rise to a map $V(\Lambda) \oplus V(\Lambda') \to M$ which is injective since $V(\Lambda)$ and $V(\Lambda')$ are distinct irreducibles. Applying $\Psi_\Lambda$, this gives a surjection $H \to H \oplus H$, which is a contradiction.

	(4) and (5) are proven in \cite{VY}, which also is a direct consequence of (1).
	\end{proof}
	\begin{cor}
	There exists an invariant hermitian form on $V_q(\lambda,\nu)$ if and only if there exists $w \in W$ such that $w \lambda = \lambda$ and $w \nu = -\overline{\nu}$ modulo $2 \pi i \log(q)^{-1} Q^\vee$.

	In particular, if $\nu$ is almost real, $w \nu = -\overline{\nu}$ in $\fh^*$.
	\end{cor}
	\begin{proof}
	The first part can be proven in the same way as \cite[Corollary 5.6]{Arano}. The second part is a consequence of Lemma \ref{lem case-by-case}.
	\end{proof}
	\begin{cor}
	If $2\Lambda'$ is almost real, the representation $\Psi_\Lambda(V_q(\Lambda'))$ has the same $K$-type multiplicity as $\Psi_\Lambda(V_1(\Lambda'))$.
	In particular, $V_q(\lambda,\nu)$ has the same $K$-type multiplicity with $V_1(\lambda,\nu)$ if $\nu$ is almost real.
	\end{cor}
	\begin{proof}
	Thanks to Lemma \ref{lem char}, the character of $V_q(\Lambda')$ is the same as $V_1(\Lambda')$. Hence we get the same factor in the resolution by the Verma modules.

	Now applying $\Psi_\Lambda$ to the resolution, we get the desired conclusion since the $K$-type multiplicities of $L_q(\lambda,\mu)$ are the same as in the classical case.
	\end{proof}
	The following lemma is a well-known result in linear algebra, so we omit the proof.
	\begin{lem}\label{lem triv}
	Let $(\cdot,\cdot)_q$ be a continuous path of hermitian forms on a finite dimensional vector space $V$ for an interval $q \in [q_0,1]$. Assume that the dimensions of annihilators are constant. Then $(\cdot,\cdot)_q$ is positive definite for all $q$ if and only if it is for some $q$.
	\end{lem}
	\begin{thm}\label{prop main}
	For almost real $\nu$, the representation $V_q(\lambda,\nu)$ is unitary if and only if $V_1(\lambda,\nu)$ is.
	\end{thm}
	\begin{proof}
	We take $-1/2(\lambda-\nu)-\rho$ to be dominant, so that $V_q(\lambda,\nu)$ is a quotient of  $L_q(\lambda,\nu) = L_\lambda$. Pick $w \in W$ such that $w \nu = -\overline{\nu}$.
	Then the invariant sesquilinear form $(\cdot,\cdot)_q : L_\lambda \times L_\lambda \to {\mathbb C}$ is given by
	\[(x,y)_q = (x,T_q^w y)_q^0,\]
	where $T_q^w$ is the intertwining operator $L_q(\lambda,\nu) \to L_q(w\lambda,w\nu) = L_q(\lambda,-\overline{\nu})$.
	Hence it varies continuously with respect to $q$.
	This, we can apply Lemma \ref{lem triv} for each $K$-type to get the desired conclusion.
	\end{proof}
	\begin{cor}
	Let $K$ be a connected simply connected compact Lie group and fix $0<q<1$.
	Then, the discrete quantum group $\widehat{K_q}$ has central property ${\rm (T)}$.
	Equivalently, the tensor category ${\rm Rep}(K_q)$ has property ${\rm (T)}$.
	\end{cor}
	\begin{proof}
	The trivial representation is $V_q(0,2\rho)$ in our parametrization. Since the set of almost real weights is open, we get the conclusion.
	\end{proof}
	We conclude this section with the case of $K=SU(n)$.
	Let us remark that for $\chi \in 2 \pi i \log(q)^{-1} P^\vee$, the module $V_q(\lambda,\nu-\chi)$ is unitary if and only if $V_q(\lambda,\nu)$ is. The following corollary is an immediate consequence of Lemma \ref{lem type A} and Theorem \ref{prop main}.
	\begin{cor}\label{cor main}
	Let $K = SU(n)$. For $(\lambda,\nu) \in P \times \fh^*$, the representation $V_q(\lambda,\nu)$ is unitary if and only if there exists $\chi \in 2 \pi i \log(q)^{-1} P^\vee$ such that $V_1(\lambda,\nu-\chi)$ is unitary.
	\end{cor}
	
\end{document}